\documentclass[a4paper,11pt]{article}
\usepackage[english]{babel}
\usepackage[centertags]{amsmath}
\usepackage{amsfonts}
\usepackage{amssymb}
\usepackage{amsthm}
 \newtheorem{theorem}{Theorem}[section]
 \newtheorem{corollary}[theorem]{Corollary}
 \newtheorem{lemma}[theorem]{Lemma}
 
 \theoremstyle{definition}
 \newtheorem{definition}[theorem]{Definition}
 \theoremstyle{remark}
 \newtheorem{remark}[theorem]{Remark}
 \theoremstyle{remark}
 
 \numberwithin{equation}{section}
\begin{document}
\begin{center}
\Large{\textbf{On graded Brown--McCoy radicals of graded rings}}\let\thefootnote\relax
\end{center}
\begin{center}
\large{\textbf{Emil Ili\'{c}-Georgijevi\'{c}}}
\end{center}
\begin{abstract}
\noindent We investigate the graded Brown--McCoy and the classical
Brown--McCoy radical of a graded ring, which is the direct sum of
a family of its additive subgroups indexed by a nonempty set,
under the assumption that the product of homogeneous elements is
again homogeneous. There are two kinds of the graded Brown--McCoy
radical, the graded Brown--McCoy and the large graded Brown--McCoy
radical of a graded ring. Several characterizations of the graded
Brown--McCoy radical are given, and it is proved that the large
graded Brown--McCoy radical of a graded ring is the largest
homogeneous ideal contained in the classical Brown--McCoy radical
of that ring.
\end{abstract}
\section{Introduction}
Graded radicals and radicals of group-graded rings have been
investigated by many authors. Particularly, the graded
Brown--McCoy and the Brown--McCoy radical (classical) of
group-graded rings have been studied recently, as well as in the
recent decades, regarding many important open problems. For
example, in \cite{lp}, results related to the Brown--McCoy radical
of a ring graded by the additive group of integers are obtained in
order to examine open problems on the Brown--McCoy radical of a
polynomial ring (see also \cite{as}). These are, on the other
hand, related to the famous K\"{o}the's Conjecture (consult also
references of \cite{lp}). In \cite{grz}, $G$-systems are
introduced, where $G$ is a group, and their Brown--McCoy radicals
are investigated, while in \cite{bs}, the graded Brown--McCoy and
the Brown--McCoy radical of $G$-graded rings are studied. For
research related to some problems on monomial rings, see
\cite{jp}. There is, however, a more general notion of a graded
ring and the aim of this paper is to study the graded Brown--McCoy
and the classical Brown--McCoy radical of such rings. Results
dealing with the Jacobson radical of these rings can be found in
\cite{ak2}; see also \cite{ak1} and \cite{ak}.
\begin{definition}[\cite{ak2,ak}]\label{sgris}
Let $R$ be a ring, and $S$ a partial groupoid, that is, a set with
a partial binary operation. Also, let $\{R_s\}_{s\in S}$ be a
family of additive subgroups of $R.$ We say that
$R=\bigoplus_{s\in S}R_s$ is $S$-\emph{graded} and $R$
\emph{induces} $S$ (or $R$ is an $S$-\emph{graded ring inducing}
$S$) if the following two conditions hold:
\begin{itemize}
    \item[$i)$] $R_sR_t\subseteq R_{st}$ whenever $st$ is defined;
    \item[$ii)$] $R_sR_t\neq0$ implies that the product $st$ is
    defined.
\end{itemize}
\end{definition}
\noindent An $S$-graded ring inducing $S$ from above is also
called a \emph{homogeneous sum} \cite{ak2,ak}. However, we will
call it simply a \emph{graded ring} in the sequel. If
$R=\bigoplus_{s\in S}R_s$ is a graded ring, the set
$A=\bigcup_{s\in S}R_s$ is called the \emph{homogeneous part of
$R,$} and elements of $A$ are called \emph{homogeneous elements of
$R.$} The previous definition applies to both associative and
nonassociative rings, but all
rings in this paper are assumed to be associative.\\
Clearly, every usual group (semigroup, groupoid) graded ring is
graded in the above sense. However, there are many other important
examples of such rings. For instance, let $A,$ $B$ be rings and
$V,$ $W$ be an $(A,B)$-bimodule and a $(B,A)$-bimodule,
respectively, and let the quadruple $(A,V,W,B)$
be a Morita context, that is, the set $R=\left(%
\begin{array}{cc}
  A & V \\
  W & B \\
\end{array}%
\right)$ of matrices forms a ring under matrix addition and
multiplication (see, for instance, \cite{gw}). Then
$R=\bigoplus_{(i,j)\in\{1,2\}^2}R_{(i,j)}$ is a graded ring in the
above sense if $R_{(i,j)}$ denotes the set of matrices of $R$
whose all entries are zero except the $(i,j)$ entry which does not
have to be zero. The following are also examples of graded rings
in the above sense: a semidirect sum of rings (see \cite{ak}),
particular case of which is the Dorroh extension of a ring, a ring
which is the direct sum of its left ideals, a path
algebra (see \cite{ak}).\\
A notion of a graded ring, equivalent to that in Definition
\ref{sgris}, was studied in \cite{cha,ha,agg} from a different
point of view. Namely, let $R$ be a graded ring and $A$ its
homogeneous part. Consider $A$ with induced partial addition and
induced multiplication from $R.$ Then $A$ is called an
\emph{anneid} \cite{cha,ha,agg}. Origins of this approach can be
found in \cite{mk}. Anneid $A$ has a partial addition since the
sum of nonzero elements $x,y\in A$ does not have to be a
homogeneous element of $R.$ However, if $x+y$ belongs to $A,$
elements $x$ and $y$ are called \emph{addable} and we write $x\#y$
\cite{cha,ha,agg}. Multiplication is defined everywhere, according
to the very definition of a graded ring. As we will see in the
next section, studying graded rings is equivalent to studying
their corresponding anneids \cite{cha,ha,agg}. We will therefore,
in this paper, conduct research on anneids, all of which are
assumed to be associative.\\
The graded Jacobson radical of an anneid is thoroughly
investigated in \cite{ha1,ha}. Similar studies on $2$-graded rings
can be found in \cite{sw}. Here, we use results of \cite{ha1,ha}
in order to introduce and investigate the \emph{graded
Brown--McCoy radical of an anneid}. Inspired by \cite{ha1,ha}, we
give two notions of this radical, the \emph{graded Brown--McCoy}
and the \emph{large graded Brown--McCoy radical of an anneid}.
Several characterizations of the graded Brown--McCoy radical of an
anneid are given, and as the main result we prove that the large
graded Brown--McCoy radical of an anneid is the homogeneous part
of the largest homogeneous ideal contained in the classical
Brown--McCoy radical of the corresponding graded ring (Theorem
\ref{theorem}).\\ Obtained results can be easily translated to the
language of graded rings, and represent generalizations of results
which hold for usual group-graded rings \cite{grz,bs}.
\section{Preliminaries}
\noindent Here, we give some notions and properties related to
graded rings described above. Everything stated in this section
can be found in more detail in \cite{agg} and references therein
(particularly \cite{cha,ha}) or in \cite{kv,mv}.
\begin{theorem}[\cite{cha,ha,agg}]
Anneid $A$ can be characterized by the following axioms:
\begin{itemize}
    \item[$a1)$] $A$ is a groupoid with respect to multiplication and
possesses an element $0$ for which we have $a0=0=0a$ for all $a\in
A;$
    \item[$a2)$] $0$ is addable with every element of $A,$ and the
relation of addibility is reflexive and almost transitive:
$(\forall a,b,c\in A)\ a\#b\wedge b\#c\wedge b\neq0\Rightarrow
a\#c;$
    \item[$a3)$] for all $0\neq a\in A,$ the set
$A(a)$ of all elements from $A,$ which are addable with $a,$ is a
commutative group with respect to addition induced by that of $A$
and such a group is called the \emph{addibility group of} $a;$
    \item[$a4)$] for all
$a,b,c\in A,$ $a\#b$ implies $ca\#cb,$ $ac\#bc,$ $c(a+b)=ca+cb,$
$(a+b)c=ac+bc.$
\end{itemize}
\end{theorem}
\noindent Clearly, $0$ is the neutral element of all addibility
groups. Also, two distinct addibility groups have only one common
element, and it is $0.$ Let $\Delta^*$ denote the set of
addibility groups of $A$ and let $\Delta=\Delta^*\cup0.$ If
$0\neq\delta\in\Delta,$ then let $A(\delta)$ denote the addibility
group of $A$ that defines $\delta.$ Otherwise, that is, if
$\delta=0,$ $A(\delta)=0.$ Then, if $\xi,\eta\in\Delta,$ we put
\[\xi\eta=\left\{%
\begin{array}{ll}
    0, & \textrm{if}\ A(\xi)A(\eta)=0; \\
    \zeta, & \textrm{if}\ A(\xi)A(\eta)\subseteq A(\zeta) \\
\end{array}%
\right.\] (see, for instance, \cite{cha}). This multiplication on
$\Delta$ is justified since it is known that if $x,y,a,b\in A,$
then $x\#y$ and $a\#b$ imply $xa\#yb$ (see, for instance,
\cite{agg}). Now, to each anneid $A,$ we may associate a graded
ring $\overline{A},$ called the \emph{linearization of $A$}, in
the following manner. We will follow the exposition from \cite{ha}
(see also \cite{cha}). Define $\overline{A}$ to be the direct sum
of $A(\delta),$ where $\delta$ runs through $\Delta^*.$ The
multiplication of $\overline{A}$ is obtained by extending linearly
the multiplication of $A.$ The set $\Delta^*$ plays the role of
$S$ from Definition \ref{sgris}. Product of two elements of
$\Delta^*$ is defined if and only if their product is distinct
from $0$ in $\Delta,$ in which case these products coincide. Then
$\overline{A}=\bigoplus_{\delta\in\Delta^*}A(\delta)=\bigoplus_{\delta\in\Delta}A(\delta)$
is a graded ring whose homogeneous part is $A.$ If $0\neq a\in A,$
then the \emph{degree of $a$} \cite{cha,ha,agg}, denoted by
$\delta(a),$ is the element of $\Delta$ such that $a\in
A(\delta(a)).$ It is assumed to be equal to $0$ if $a=0.$ If
$a\neq0,$ then of course, $A(\delta(a))=A(a).$ Hence
$\delta(ab)=\delta(a)\delta(b)$ if $ab\neq0,$ and
$\delta(a)(\delta(b)\delta(c))=(\delta(a)\delta(b))\delta(c)$ if
$abc\neq0$ \cite{cha,ha,agg}. If $\delta$ is an idempotent element
of $\Delta^*,$ that is, if $\delta\delta=\delta,$ then clearly,
$A(\delta)$ is a ring with respect to operations induced from $A.$
\begin{definition}[\cite{cha,ha,agg}]
An anneid $A$ is said to be \emph{regular} if each of the
conditions $ac\#bc$ or $ca\#cb$ implies $a\#b$ for $ac\neq0,$
$bc\neq0,$ $ca\neq0,$ $cb\neq0,$ $a,b,c\in A.$
\end{definition}
\noindent In the language of graded rings, regularity is
equivalent to the notion of \emph{cancellativity} \cite{ak2,ak}.\\
A nonempty subset $I$ of an anneid $A$ is called a \emph{right
ideal of $A$} \cite{cha,ha,agg} if: $i)$ for all $a,b\in I,$ for
which $a\#b,$ we have $a-b\in I;$ $ii)$ for all $a\in I$ and $x\in
A,$ we have $ax\in I.$ A \emph{left ideal of an anneid $A$} is
defined similarly, and a subset of $A$ which is a left and a right
ideal is called an \emph{ideal} or a \emph{two-sided ideal}. If a
nonempty subset $I$ of $A$
satisfies the first condition only, then $I$ is called a \emph{subanneid} of $A.$\\
If $A$ and $A'$ are two anneids, then the mapping $f:A\to A'$ is
called a \emph{homomorphism} \cite{cha,ha,agg} if for all $a,b\in
A:$
\begin{itemize}
    \item[$i)$] $a\#b\Rightarrow f(a)\#f(b)\wedge f(a+b)=f(a)+f(b);$
    \item[$ii)$] $f(ab)=f(a)f(b);$
    \item[$iii)$] $f(a)\#f(b)\wedge f(a),f(b)\neq0\Rightarrow a\#b.$
\end{itemize}
Let $I$ be an ideal (left, right, two-sided) of an anneid $A$ and
let us define $a\sim b$ if and only if either both $a$ and $b$
belong to $I$ or $a\#b$ and $a-b\in I.$ It is easy to verify that
$\sim$ is an equivalence relation on $A.$ We also say that $a$ and
$b$ are congruent modulo $I$ and write $a\sim b\mod I.$ Denote the
set $A/\sim$ by $A/I.$ Clearly, $[a]_\sim=a+I=\{a+x\ |\ x\in
I\wedge a\#x\}.$ Now, let $I$ be a two-sided ideal, and let
$f:A\to A/I$ be the mapping defined by $f(a)=[a]_\sim.$ Put
$[a]_\sim[b]_\sim:=f(ab),$ $a,b\in A.$ We say that $[a]_\sim$ and
$[b]_\sim$ are addable and write $[a]_\sim\#[b]_\sim$ if $a\#b.$
If $[a]_\sim\#[b]_\sim,$ we put $[a]_\sim+[b]_\sim:=f(a+b).$ Then
$A/I$ becomes an anneid, called the \emph{factor anneid}
\cite{cha,ha,agg} (see also \cite{ak3}). Also, $f$ is a
homomorphism of anneids and if $A$ is regular, $A/I$ is regular
too. Isomorphism theorems for anneids are analogous to those for
ordinary rings.
\section{The graded Brown--McCoy radical}
\noindent Let $A$ be an anneid. We will fix some notation. If $I$
is a right ideal of $A,$ then $\check{I}$ denotes the largest
ideal of $A$ contained in $I.$ The class of simple regular anneids
with unity, that is, of regular anneids with unity whose only
ideals are $0$ and the anneid itself, will be denoted by
$\mathcal{M}.$ Also, if $a$ is an element of $A,$ $\langle
a\rangle$ denotes the ideal of $A$ generated by $a,$ that is, the
set comprised of all elements of the form
$na+\sum_{\textrm{finite}}ax_i+\sum_{\textrm{finite}}y_ia+\sum_{\textrm{finite}}x'_iay'_i,$
where $n$ is an integer, $x_i,y_i,x'_i,y'_i\in A,$ and where all
the summands are mutually addable \cite{cha,ha,agg}.\\
Following \cite{aa}, it makes sense to define the \emph{graded
Brown--McCoy radical $G(A)$ of an anneid} $A$ to be the
intersection of all maximal right ideals $I$ of $A$ such that
$A/\check{I}\in\mathcal{M}.$
\begin{remark}
The linearization $\overline{G(A)}=\bigoplus_{\delta\in
\Delta^*}G(A)\cap A(\delta)$ of $G(A)$ is a homogeneous (graded)
ideal of a graded ring $\overline{A}.$ This graded ideal will be
referred to as a \emph{graded Brown--McCoy radical of a graded
ring}. We could have simply called $G(A)$ the Brown--McCoy radical
of an anneid $A,$ just as it was done in \cite{ha1,ha} in the case
of the graded Jacobson radical of an anneid, which was called just
radical in accordance with \cite{nj}. We added the adjective
`graded' in order to avoid confusion with the classical
Brown--McCoy radical of a ring while discussing relations between
these notions.
\end{remark}
\noindent A right ideal $I$ of an anneid $A$ is called a
\emph{modular right ideal} if there exists $e\in A$ such that
$ea\sim a\mod I$ for all $a\in A$ \cite{ha1}. If $I$ is a proper
modular right ideal, then the degree of $e$ is an idempotent
element of $\Delta^*$ \cite{ha1,ha}. This notion served in a
description of the graded Jacobson radical of a regular anneid in
\cite{ha1,ha} by analogy with classical rings. In order to do the
same with the graded Brown--McCoy radical of a regular anneid, we
introduce the following natural notion.
\begin{definition}
An ideal $I$ of an anneid $A$ is called \emph{modular} if there
exists an element $e\in A$ such that for all $a\in A,$ $ea\sim
a\mod I$ and $ae\sim a\mod I.$ Element $e$ is called a \emph{unity
modulo} $I.$ We also say that an ideal $I$ is \emph{modular with
respect to $e.$}
\end{definition}
\begin{remark}\label{id}
If $I$ is a proper modular ideal of an anneid $A$ with respect to
$e,$ then $\delta(e),$ the degree of $e,$ is an idempotent element
of $\Delta^*,$ that is, $\delta(e)\delta(e)=\delta(e).$ We proceed
as in \cite{ha}. Indeed, since $e$ is a unity modulo $I,$ for all
$a\in A,$ either $ea,$ $a\in I$ or $ea\#a$ and $ea-a\in I,$ and,
either $ae,$ $a\in I$ or $ae\#a$ and $ae-a\in I.$ Since $I$ is
proper, $e\notin I.$ Also, $e^2\#e$ and $e^2-e\in I.$ Since
$e\notin I,$ we have that $e^2\neq0.$ This and the fact that $e^2$
and $e$ belong to the same addibility group, that is, $e^2\#e,$
imply the claim.
\end{remark}
\noindent The proof of the following lemma is included for the
sake of completeness.
\begin{lemma}
An ideal $I$ of an anneid $A$ is a maximal modular ideal if and
only if $A/I$ is a simple anneid with unity.
\end{lemma}
\begin{proof}
Let $I$ be a maximal modular ideal and $e$ a unity modulo $I.$ If
$a+I$ is an arbitrary element of $A/I,$ then, since $e$ is a unity
modulo $I,$ $ea+I=a+I=ae+I.$ Hence $(e+I)(a+I)=(a+I)(e+I)=a+I,$
which proves that $e+I$ is a unity of $A/I.$ Clearly, $A/I$ is a
simple anneid, since $I$ is a maximal
ideal.\\
Conversely, if $A/I$ is a simple anneid with unity $e+I,$ then,
for all $a\in A,$ we have $(e+I)(a+I)=(a+I)(e+I)=a+I,$ that is,
$ea+I=a+I=ae+I.$ Hence $ea\sim a\mod I$ and $ae\sim a\mod I.$
Therefore $I$ is a modular ideal. Also, since $A/I$ is simple, $I$
is a maximal ideal.
\end{proof}
\begin{corollary}\label{mod}
If $A$ is a regular anneid, then $G(A)$ coincides with the
intersection of all maximal right ideals $I$ of $A$ such that the
ideals $\check{I}$ are modular.
\end{corollary}
\begin{proof}
If $I$ is a maximal right ideal of $A$ such that $\check{I}$ is
modular, then $A/\check{I}$ is a simple anneid with unity.
Moreover, $A/\check{I}$ is a regular anneid since $A$ is regular.
Conversely, if $I$ is a maximal right ideal of $A$ such that
$A/\check{I}$ is a regular simple anneid with unity, then
$\check{I}$ is a modular ideal of $A.$
\end{proof}
\begin{remark}\label{mod1}
If $I$ is a maximal right ideal such that $\check{I}$ is modular
with respect to $e,$ then $e$ is also a unity modulo $I,$ that is,
$ea\sim a\mod I$ and $ae\sim a\mod I$ for all $a\in A.$
\end{remark}
\noindent It is sometimes more pleasable to describe radicals of
rings by their modules. It is shown in \cite{ar} that any special
radical of an associative ring may be defined by some class of
modules. For analogous results on special radicals of group-graded
rings, see \cite{ib}. Here, we aim to describe the graded
Brown--McCoy radical of an anneid $A$ by the means of right
$A$-\emph{moduloids}, and so in the next paragraph we recall the notion of a moduloid.\\
If a group is the direct sum of a family of its subgroups, indexed
by a nonempty set, then such a group is called a \emph{graded
group} \cite{agg}. Let $\overline{M}=\bigoplus_{d\in
D}\overline{M}_d$ be a commutative graded group and
$\overline{A}=\bigoplus_{\delta\in\Delta}\overline{A}_\delta$ a
graded ring, where $D$ and $\Delta$ are nonempty sets. Moreover,
suppose that $\overline{M}$ is a right $\overline{A}$-module. A
right $\overline{A}$-module $\overline{M}$ is called \emph{graded}
\cite{cha,ha,agg} if for all $d\in D$ and $\delta\in\Delta$ there
exists $t\in D$ such that
$\overline{M}_d\overline{A}_\delta\subseteq\overline{M}_t.$ A
right $A$-\emph{moduloid} \cite{cha,ha,agg} $M$ is just the
homogeneous part of $\overline{M},$ that is, $M=\bigcup_{d\in
D}\overline{M}_d,$ with induced partial addition from
$\overline{M}$ and induced outer operation $M\times A\to M$ from
$\overline{M}\times\overline{A}\to\overline{M},$ where $A$ is the
corresponding anneid of $\overline{A}.$ Similarly to the case of
anneids, to an $A$-moduloid $M$ one can associate a graded
$\overline{A}$-module $\overline{M},$ called the
\emph{linearization of an $A$-moduloid} $M,$ where $\overline{A}$
is the linearization of an anneid $A$ (see, for instance, \cite{kv}).\\
A right $A$-moduloid $M$ is said to be \emph{regular}
\cite{cha,ha,agg} if $xa\#xb,$ $xa,xb\neq0,$ implies
$a\#b,$ where $x\in M,$ $a,b\in A.$\\
If $M$ and $M'$ are two right $A$-moduloids, then the mapping
$f:M\to M'$ is called a \emph{homomorphism} \cite{cha,ha,agg} if
for all $x,y\in M,$ $a\in A:$ $i)$ $x\#y\Rightarrow
f(x)\#f(y)\wedge f(x+y)=f(x)+f(y);$ $ii)$ $f(xa)=f(x)a;$ $iii)$
$f(x)\#f(y)\wedge f(x),f(y)\neq0\Rightarrow x\#y.$\\
Submoduloids are defined as usual (see \cite{cha}). A right
$A$-moduloid $M$ is said to be \emph{irreducible} \cite{ha} if
$MA\neq0$ and only submoduloids of $M$ are $0$ and $M.$ If $x$ is
an element of a right $A$-moduloid, then $(0:x)$ denotes the set
$\{a\in A\ |\ xa=0\},$ which is a right ideal of an anneid $A$
\cite{ha}. Correspondingly, the annihilator of a right
$A$-moduloid $M,$ that is, the set $\{a\in A\ |\ Ma=0\},$ will be
denoted by $(0:M)_A$ or just by $(0:M)$ (see \cite{ha}). If $A$ is
an anneid and $I$ a right ideal of $A,$ then $A/I$ is a right
$A$-moduloid with respect to $(a+I)+(b+I)=(a+b)+I$ if $a\#b,$ and
$(a+I)c=ac+I$ for $a+I, b+I\in A/I,$ $c\in A$ (see \cite{ha}).
Also, $f:A\to A/I,$ defined by $f(a)=a+I$ for $a\in A,$ is a
homomorphism of right $A$-moduloids \cite{ha}.\\
The notion of a simplicity is commonly used to mean that an
algebraic structure in hand has only trivial substructures.
Inspired by the notion from \cite{ib} for usual group-graded
modules, we give simplicity a different meaning in the following
definition.
\begin{definition}\label{definition}
An irreducible right $A$-moduloid $M$ over an anneid $A$ is called
\emph{simple} if $MA\neq0$ and if for all ideals $I$ of $A$ for
which $MI\neq0,$ there exists
$b\in\overline{I}=\bigoplus_{\delta\in \Delta^*}A(\delta)\cap I$
such that $xb=x$ for all $x\in M.$
\end{definition}
\noindent We will soon see, after proving two lemmas, that the
graded Brown--McCoy radical of a regular anneid $A$ coincides with
the intersection of annihilators of all regular simple right
$A$-moduloids. Therefore, inspired by the notion of the large
graded Jacobson radical of an anneid from \cite{ha1,ha}, it makes
sense to introduce the following notion.
\begin{definition}
The intersection of annihilators of all simple right
$A$-modul\-oids, which are not necessarily regular, is called the
\emph{large graded Brown--McCoy radical of an anneid} $A$ and is
denoted by $G_l(A).$
\end{definition}
\begin{lemma}\label{remark}
Let $M$ be a regular simple right $A$-moduloid. Then, if $I$ is an
ideal of $A$ such that $MI\neq0,$ then there exists $b\in I$ such
that $xb=x$ for all $x\in M,$ and, moreover, the degree of $b$ is
an idempotent element of $\Delta^*.$
\end{lemma}
\begin{proof}
Let $M$ be a regular simple right $A$-moduloid, and $I$ an ideal
of $A$ such that $MI\neq0.$ By definition, there exists
$b\in\overline{I}$ such that $xb=x$ for all $x\in M.$ Then we have
$b=b_1+\dots+b_n$ for some natural number $n,$ and where $b_i\in
I,$ $i\in\{1,\dots,n\}.$ For an arbitrary $x\in M,$
$x=xb=x(b_1+\dots+b_n)=xb_1+\dots+xb_n.$ This means that $xb_i$
are mutually addable, and since $M$ is regular, $b_i$ are mutually
addable. Therefore $b=b_1+\dots+b_n\in I.$ Let us now prove that
the degree of $b$ is an idempotent element of $\Delta^*.$ Indeed,
since $xb=x$ for all $x\in M,$ we have $xb^2=xb\neq0,$ for
$x\neq0.$ Now, the regularity of $M$ implies $b^2\#b,$ and the
claim follows.
\end{proof}
\begin{lemma}
If $M$ is a regular simple right $A$-moduloid, then $M\cong A/I,$
where $I$ is a maximal right ideal of $A$ such that $\check{I}$ is
a modular ideal of $A.$ Conversely, if $I$ is a maximal right
ideal of $A$ such that $\check{I}$ is a modular ideal of $A,$ then
$A/I$ is a simple right $A$-moduloid.
\end{lemma}
\begin{proof}
Let $M$ be a regular simple right $A$-moduloid. Then, since
$MA\neq0,$ and since $M$ is regular, there exist $0\neq x\in M$
and $a\in A$ such that $xa\neq0.$ Therefore $M\langle
a\rangle\neq0.$ Since $M$ is simple and regular, Lemma
\ref{remark} implies that there exists $a'\in\langle a\rangle$
such that $ma'=m$ for all $m\in M.$ Let $x\neq0$ be an element of $M.$
Since $M$ is irreducible, we know from \cite{ha} that $M=xA,$ by analogy with the 
case of classical modules \cite{nj}. On the other hand, 
from \cite{ha} we know that
$xA\cong A/(0:x)$ by the first isomorphism theorem for moduloids
(for the theorem, see \cite{cha,ha}). Therefore, according to \cite{ha},
$(0:x)$ is a maximal modular right ideal of $A,$ that is, there
exists $e\in A$ such that $ea\sim a \mod (0:x)$ for all $a\in A.$
Denote $(0:x)$ by $I.$ It suffices now to prove that $\check{I}$
is a modular ideal of $A.$ Since $I$ is a modular right ideal of
$A,$ we have that $\check{I}=\{a\in A\ |\ Aa\subseteq I\},$ by
analogy with classical rings (see, for instance, \cite{bb}). We
already know that there exists $a'\in A$ such that $ma'=m$ for
every $m\in M,$ and particularly, $xa'=x.$ Also, for every $b\in
A,$ and every $c\in A,$ we have $x(cba')=(xcb)a'=xcb.$ Then, if
$x(cba')=xcb\neq0,$ the regularity of $M$ implies $cba'\#cb$ and
$ba'\#b.$ Hence $cba'-cb\in I$ and $ba'-b\in\check{I}.$ If
$x(cba')=xcb=0,$ then, of course, both $cba'$ and $cb$ belong to
$I,$ meaning that $ba',$ $b\in\check{I}.$ Also, for every $b\in
A,$ and every $c\in A,$ $x(ca'b)=(xca')b=xcb,$ which again, by the
regularity of $M,$ implies $ca'b-cb\in I$ if $x(ca'b)=xcb\neq0,$
and $ca'b,cb\in I$ if $x(ca'b)=xcb=0.$ Therefore $\check{I}$ is a
modular ideal of $A.$\\
Conversely, if $I$ is a maximal right ideal of $A$ such that
$\check{I}$ is modular, it is clear that $M=A/I$ is an irreducible
right $A$-moduloid. Let us recall, if $\check{I}$ is a modular
ideal with respect to $e,$ then $e$ is a unity modulo $I.$ If $K$
is an ideal of $A$ such that $MK\neq0,$ then $MK=M.$ Hence there
exists $f\in K$ such that $e+I=f+I,$ and so for all $a+I\in M,$
$af+I=a+I.$ Therefore $M$ is a simple right $A$-moduloid.
\end{proof}
\begin{corollary}\label{module}
The graded Brown--McCoy radical of a regular anneid $A$ coincides
with the intersection of annihilators of all regular simple right
$A$-moduloids.
\end{corollary}
\begin{remark}
If $A$ is a regular anneid, notice that we have $G_l(A)\subseteq
G(A).$
\end{remark}
\noindent Let us now deal with the elementwise characterization of
the graded Brown--McCoy radical of a regular anneid, inspired by a
similar study on the graded Jacobson radical of a regular anneid
from \cite{ha1,ha}.\\
The following lemma follows the proof of the analogous lemma for
modular right ideals of regular anneids from \cite{ha}.
\begin{lemma}
If $I$ is a proper modular ideal of a regular anneid $A,$ then all
unities modulo $I$ have the same degree. Moreover, this degree is
an idempotent element of $\Delta^*.$
\end{lemma}
\begin{proof}
Let $f:A\to A/I$ be the canonical mapping, and $e,$ $e'$ be two
unities modulo $I.$ Then for all $a\in A,$
$f(a)=f(ea)=f(e'a)=f(ae)=f(ae').$ Hence in particular,
$\bar{0}\neq f(ea)\#f(e'a)\neq\bar{0}$ if $a\notin I.$ Since $f$
is a homomorphism, it follows that $ea\#e'a.$ Now, the regularity
of $A$ implies $e\#e'.$ Therefore, according to Remark \ref{id},
$e$ and $e'$ have the same degree, which is an idempotent element
of $\Delta^*.$
\end{proof}
\begin{definition}
The degree of all unities modulo a proper modular ideal $I$ of a
regular anneid is called the \emph{degree of $I.$}
\end{definition}
\noindent It is stated in \cite{ha1} and proved in \cite{ha} that
there exists a one-to-one correspondence between the maximal
modular right ideals of a regular anneid $A$ of degree $\epsilon$
and the maximal modular right ideals of $A(\epsilon).$ Here, we
prove the same for maximal modular ideals if we additionally assume that 
$\Delta$ is such that the product of nonidempotent elements cannot be a nonzero 
idempotent. We assume the same in Theorem~\ref{gregular}, Theorem~\ref{thm2},
Corollary~\ref{corollary} and Theorem~\ref{thm1}.
\begin{theorem}\label{thm}
Let $A$ be a regular anneid and $\epsilon$ an idempotent element
of $\Delta^*.$ If $I$ is a maximal right ideal of $A$ such that
$\check{I}$ is a modular ideal of $A$ of degree $\epsilon,$ then
$I_\epsilon=I\cap A(\epsilon)$ is a maximal right ideal of
$A(\epsilon)$ such that $\check{I_\epsilon}$ is a modular ideal of
$A(\epsilon).$ If $I$ is a maximal right ideal of $A(\epsilon)$
such that $\check{I}$ is modular, then $K=\{x\in A\ |\ xA\cap
A(\epsilon)\subseteq I\}$ is a maximal right ideal of $A$ such
that $\check{K}$ is a modular ideal of $A$ of degree $\epsilon.$
\end{theorem}
\begin{proof}
Let $I$ be a maximal right ideal of $A(\epsilon)$ such that $\check{I}$ is a
modular ideal of $A(\epsilon)$ with respect to $e.$ Since $A$ is regular, according
to \cite{ha1,ha} we have that $K=\{x\in A\ |\ xA\cap A(\epsilon)\subseteq I\}$
is a maximal modular right ideal of $A$ with respect to $e,$ and
$K\cap A(\epsilon)=I.$ It is clear that 
$\check{K}\cap A(\epsilon)\subseteq\check{I}.$ Let $x\in\check{I},$ and let
$a\in A.$ If $b\in A$ is such that $0\neq axb\in A(\epsilon),$ then, by our
assumption on $\Delta,$ we have that both $a$ and $b$ belong to $A(\epsilon).$
Since $x\in\check{I},$ it follows that $axb\in\check{I}.$ Therefore
$\check{I}\subseteq\check{K}\cap A(\epsilon).$ Thus 
$\check{K}\cap A(\epsilon)=\check{I},$ and so $\check{K}\neq0.$ It follows that
$\check{K}$ is a maximal ideal of $A.$ Let $a\in A.$ If $a\in\check{K},$ then
$ea,$ $ae\in\check{K}.$ If $a\notin\check{K},$ then there exist $x,$ $y\in A$
such that $xay\in A(\epsilon)$ but $xay\notin I.$ In particular, $xay\neq0.$
Since $A$ is regular, our assumption on $\Delta$ implies that 
$\delta(x)=\delta(y)=\delta(a)=\epsilon.$ By assumption, $\check{I}$ is a 
modular ideal of $A(\epsilon)$ with respect to $e.$ Therefore $ea\#a,$ $ae\#a,$
and $ea-a,$ $ae-a\in\check{I}\subseteq\check{K}.$
Since $\check{K}$ is a maximal ideal of $A,$ it follows that $A/\check{K}$ is
a regular simple anneid with unity $e+\check{K}.$ 
Therefore $\check{K}$ is modular.\\
Conversely, if $I$ is a maximal right ideal of $A$ such that
$\check{I}$ is modular, then clearly, $I_\epsilon=I\cap
A(\epsilon)$ is a modular right ideal of $A(\epsilon)$ such that
$\check{I_\epsilon}$ is modular. Also, according to \cite{ha1,ha},
$I_\epsilon$ is maximal.
\end{proof}
\begin{definition}
An element $x$ of an anneid $A$ is said to be \emph{$G$-regular}
if there exists no proper ideal $I$ of $A$ such that $x$ is a
unity modulo $I.$ An anneid (ideal) is called \emph{$G$-regular}
if every of its elements is $G$-regular.
\end{definition}
\noindent The following theorem is an analogue of the
characterization of quasi-regular elements of a regular anneid
from \cite{ha1,ha} with the same proving technique.
\begin{theorem}\label{gregular}
Let $A$ be a regular anneid. An element $x\in A$ is $G$-regular if
and only if one of the following two conditions is satisfied:
\begin{itemize}
    \item[$i)$] $\delta(x)$ is not an idempotent element of $\Delta^*;$
    \item[$ii)$] $\epsilon=\delta(x)$ is an idempotent element of $\Delta^*$ and
    $x$ is a $G$-regular element of the ring $A(\epsilon).$
\end{itemize}
\end{theorem}
\begin{proof}
According to Remark \ref{id}, every unity modulo a proper modular
ideal has the degree which is an idempotent element of $\Delta^*.$
Hence, if $x\in A$ satisfies $i),$ then $x$ is $G$-regular. If
$x\in A$ is such that $\epsilon=\delta(x)$ is an idempotent
element of $\Delta^*,$ we claim that $x$ is $G$-regular if and
only if $x$ is $G$-regular in $A(\epsilon).$ It suffices to prove
that $x\in A$ is a unity modulo a maximal modular ideal of $A$ if
and only if it is a unity modulo a maximal modular ideal of
$A(\epsilon).$ However, this follows from Theorem \ref{thm}.
\end{proof}
\begin{theorem}\label{thm2}
Let $A$ be a regular anneid. Then $G(A)=A$ if and only if $A$ is a
$G$-regular anneid.
\end{theorem}
\begin{proof}
Let $G(A)=A.$ According to Remark \ref{mod1}, no element of $A$
can be a unity modulo a proper ideal, whence every element of $A$
is a $G$-regular element.\\
Conversely, suppose that every element of a regular anneid $A$ is
$G$-regular. Take any right ideal $I$ of $A$ such that $I\neq A.$
Suppose that $A/\check{I}$ has a unity $a+\check{I},$
$a\notin\check{I}.$ Then $a^2+\check{I}=a+\check{I}.$ This and the
fact that $a\notin \check{I}$ imply $0\neq a^2\#a.$ Hence
$\epsilon=\delta(a)$ is an idempotent element of $\Delta^*.$
According to the previous theorem, $a$ is a $G$-regular element of
$A(\epsilon),$ that is, $a$ belongs to
$G(a):=\{ax-x+\sum_{i=1}^n(y_iaz_i-y_iz_i)\ |\ x,y_i,z_i\in
A(\epsilon),\ n\ \textrm{a natural number}\}$ (see, for instance,
\cite{gw}). However, for any element
$ax-x+\sum_{i=1}^n(y_iaz_i-y_iz_i)$ from $G(a),$ we have
$(a+\check{I})(x+\check{I})-(x+\check{I})+\sum_{i=1}^n((y_i+\check{I})(a+\check{I})(z_i+\check{I})-(y_i+\check{I})(z_i+\check{I}))=0+\check{I},$
since $a+\check{I}$ is by assumption a unity of $A/\check{I}.$
Hence $ax-x+\sum_{i=1}^n(y_iaz_i-y_iz_i)\in \check{I}.$ Therefore
$G(a)\subseteq \check{I},$ and particularly, $a\in \check{I},$ a
contradiction. We proved that there is no proper right ideal $I$
of $A$ such that $A/\check{I}$ is a simple regular anneid with
unity, and so $A$ is a graded Brown--McCoy radical anneid, that
is, $G(A)=A.$
\end{proof}
\begin{corollary}\label{corollary}
Let $A$ be a regular anneid. Then $G(A)$ is a $G$-regular ideal
that contains all $G$-regular ideals of $A.$
\end{corollary}
\noindent If $A$ is an anneid and if $\epsilon$ is an idempotent
element of $\Delta^*,$ then we will denote the classical
Brown--McCoy radical of the ring $A(\epsilon)$ by
$G(A(\epsilon)).$ The proof of the following theorem follows
Halberstadt's proof of the equality $J(A(\epsilon))=J(A)\cap
A(\epsilon),$ where $J(A(\epsilon))$ denotes the classical
Jacobson radical of the ring $A(\epsilon),$ and $J(A)$ the graded
Jacobson radical of a regular anneid $A$ \cite{ha}. We include it
for completeness.
\begin{theorem}\label{thm1}
If $A$ is a regular anneid, then for every idempotent element
$\epsilon$ of $\Delta^*,$ we have $G(A(\epsilon))=G(A)\cap
A(\epsilon).$
\end{theorem}
\begin{proof}
Ideal $G(A)\cap A(\epsilon)$ is a $G$-regular ideal of
$A(\epsilon),$ according to Theorem \ref{gregular} and the
previous corollary. Hence, according to \cite[Corollary
4.8.3]{gw}, $G(A)\cap A(\epsilon)$ is
contained in $G(A(\epsilon)).$\\
Conversely, let $x\in G(A(\epsilon))$ and suppose that $x\notin
G(A).$ Then, according to Theorem \ref{thm} and Corollary
\ref{mod}, there exists $y\in A$ such that $\epsilon'=\delta(xy)$
is an idempotent element of $\Delta^*,$ but $xy\notin
G(A(\epsilon')).$ Clearly, $\epsilon\neq \epsilon'.$ Since
$\epsilon\delta(y)=\epsilon'$ and $A$ is a regular anneid,
$\epsilon,\epsilon'$ and $\delta(y)$ are mutually distinct. Notice
that in any ring $R,$ the set $I=\{r\in R\ |\ RrR=0\}$ is an ideal
of $R$ and $I^3=0,$ which implies that $I$ is contained in the
Jacobson radical of $R,$ and hence, $I$ is contained in the
Brown--McCoy radical of $R.$ In our case, $xy\notin
G(A(\epsilon')),$ and so there exists $z\in A(\epsilon')$ such
that $zxy\neq0.$ This implies
$\epsilon'(\epsilon\delta(y))=(\epsilon'\epsilon)\delta(y),$ and
so
$(\epsilon'\epsilon)\delta(y)=\epsilon'^2=\epsilon'=\epsilon\delta(y)\neq0.$
Therefore, since $A$ is regular,
$\epsilon'\epsilon=\epsilon=\epsilon^2\neq0.$ Again, by the
regularity of $A,$ we have $\epsilon'=\epsilon,$ a contradiction.
\end{proof}
\begin{remark}
One cannot discard the assumption made on $\Delta.$ Namely, according to 
Example in \cite{grz}, page 352, there exists a 
simple $C_2$-graded ring $R$ without unity such that $R_e$ is not a Brown--McCoy 
radical ring. Since a simple graded ring is graded simple, it follows that 
$A=R_e\cup R_g$ is a graded Brown--McCoy radical anneid. Therefore, 
$G(R_e)\subsetneqq G(A)\cap R_e.$ (Here, $A(e)=R_e.$)
\end{remark}
\noindent However, the first statement of Theorem~\ref{thm} is true in general, and 
implies the following result.
\begin{theorem}\label{theoreminc}
Let $A$ be a regular anneid and let $\epsilon$ be a nonzero idempotent element 
of $\Delta.$ Then $G(A(\epsilon))\subseteq G(A)\cap A(\epsilon).$
\end{theorem}
\begin{proof}
Let $I$ be a maximal right ideal of $A$ such that $\check{I}$ is a modular ideal
of $A$ of degree $\epsilon.$ Also, let $e$ be a unity modulo $\check{I}.$ Then
$I$ is a maximal modular right ideal of $A$ with respect to $e.$ Therefore
for every $x\in A$ we either have that $ex,$ $x\in I$ or $ex\#x$ and $ex-x\in I.$
Let $x\in A(\phi),$ where $\phi$ is a nonzero idempotent element of $\Delta$
distinct from $\epsilon.$ Then, since $A$ is regular, we must have that $x\in I.$
Hence $I\cap A(\phi)=A(\phi).$ On the other hand, according to the first statement 
of Theorem~\ref{thm}, we have that $I_\epsilon=I\cap A(\epsilon)$ is a maximal 
right ideal of $A(\epsilon)$ such that $\check{I_\epsilon}$ is a modular ideal of 
$A(\epsilon)$ with respect to $e.$ Therefore, 
$G(A)\cap A(\epsilon)=\bigcap_{\check{I}\in\mathcal{M}}I\cap A(\epsilon)$
equals the intersection of maximal right ideals $I$ of $A,$ such that $\check{I}$
are modular ideals of $A$ of degree $\epsilon,$ with $A(\epsilon).$ Hence
$G(A(\epsilon))\subseteq G(A)\cap A(\epsilon).$
\end{proof}
\indent In case $R=\bigoplus_{\delta\in\Delta}R_\delta$ is a strongly graded ring
(see \cite{ak}) with unity, where $\Delta$ is a finite group with identity $\epsilon,$ 
it is known from \cite{jp} that $G(R_\epsilon)=G(R)\cap R_\epsilon,$ where $G$ 
denotes the classical Brown--McCoy radical of a ring. In the proof of this result, the 
fact that unity element of $R$ belongs to $R_\epsilon$ is essential.\\
\indent According to \cite{ha}, if $A$ is a regular anneid, and $\overline{A}$ its
linearization, then $\overline{A}$ is a graded ring with unity $1$ if and only if the
following hold:
\begin{itemize}
\item[$i)$] For every nonzero idempotent element $\epsilon\in\Delta,$ the ring
$A(\epsilon)$ is with unity $1_\epsilon;$
\item[$ii)$] For every $x\in A$ there exist idempotent elements $\epsilon,$ 
$\phi\in\Delta^*$ such that $1_\epsilon x=x=x1_\phi;$
\item[$iii)$] $\Delta^*$ contains only a finite number of idempotent elements.
\end{itemize}
Moreover, $1=\sum_\epsilon1_\epsilon.$ Therefore, if $A$ is a regular anneid
with unity $1,$ then $\Delta^*$ contains exactly one idempotent element
$\epsilon$ and $1\in A(\epsilon).$\\
\indent If $A$ is a regular anneid, by $J(A)$ we denote \emph{the graded Jacobson
radical} \cite{ha1} of $A.$
\begin{theorem}
Let $A$ be a regular anneid with unity such that $\overline{A}$ is strongly graded.
Also, let us assume that $J(A)\cap A(\delta)$ is properly contained in $A(\delta)$
for every $\delta\in\Delta^*.$ If $\epsilon$ is the idempotent element of 
$\Delta^*,$ then $G(A(\epsilon))=G(A)\cap A(\epsilon).$
\end{theorem}
\begin{proof}
Since $A(\delta)\nsubseteq J(A)$ for every $\delta\in\Delta^*,$ it follows by
\cite{ha1,ha} that for every $\delta\in\Delta^*$ there exists a unique 
$\delta^{-1}\in\Delta^*$ such that $\delta\delta^{-1}=\epsilon.$ Since 
$\epsilon\delta=\delta\epsilon=\delta$ for every $\delta\in\Delta,$ and since
$A$ is regular, it follows easily that $\delta\delta^{-1}=\delta^{-1}\delta=\epsilon.$
Now, as in the case of group-graded rings, if $I$ is an ideal of $A,$ then 
$I_\epsilon=I\cap A(\epsilon)$ is an ideal of $A(\epsilon)$ and
$A(\delta)I_\epsilon A(\delta^{-1})=I_\epsilon$ for every $\delta\in\Delta^*.$ 
We refer to such ideals of $A(\epsilon)$ as $\Delta^*$-invariant. Also,
it can be proved that $I$ is an ideal of $A$ if and only if 
$I=AI_\epsilon=I_\epsilon A.$ 
If $I$ is an ideal of $A(\epsilon),$ then the mapping $I\mapsto IA$ defines
a one-to-one correspondence between the maximal $\Delta^*$-invariant ideals of 
$A(\epsilon)$ and the maximal ideals of $A.$ Since $A$ is with unity, if
$I$ is a maximal ideal of $A(\epsilon),$ then $A(\delta)IA(\delta^{-1})$ is a 
maximal ideal of $A(\epsilon)$ for every $\delta\in\Delta^*,$ and
$\bigcap_{\delta\in\Delta^*}A(\delta)IA(\delta^{-1})$ is a maximal
$\Delta^*$-invariant ideal of $A(\epsilon),$ just as in the proof of Proposition~2
in \cite{jp}. It follows that $G(A(\epsilon))$ equals the intersection of all maximal
$\Delta^*$-invariant ideals of $A(\epsilon).$ Therefore
$G(A(\epsilon))=G(A)\cap A(\epsilon).$
\end{proof}
\noindent Next we prove that the homogeneous part of the largest
homogeneous ideal contained in the classical Brown--McCoy radical
of a graded ring coincides with the large graded Brown--McCoy
radical of the corresponding anneid. For a similar result on the
Jacobson radical, see \cite{ha1,ha}.
\begin{theorem}\label{theorem}
Let $A$ be an anneid, $\overline{A}$ its linearization, and
$G(\overline{A})$ the classical Brown--McCoy radical of the ring
$\overline{A}.$ Then $G_l(A)=G(\overline{A})\cap A.$
\end{theorem}
\begin{proof}
Let $\overline{M}$ be a simple right $\overline{A}$-module. Then
$\overline{M}$ may be viewed as a simple right $A$-moduloid.
Indeed, $\overline{M}$ is obviously a right $A$-moduloid and
$\overline{M}A\neq0.$ Also, let $I$ be an ideal of $A$ such that
$\overline{M}I\neq0.$ Suppose that for all $a\in\overline{I}$
there exists $y\in\overline{M}$ such that $ya\neq y.$ On the other
hand, $\overline{M}\,\overline{I}\neq0$ and so there exists
$b\in\overline{I}$ such that $xb=x$ for every $x\in\overline{M}.$
In particular, $yb=y,$ a contradiction. Hence $\overline{M}$ is a
simple right $A$-moduloid. Also, it is obvious that
$(0:\overline{M})_A\subseteq(0:\overline{M})_{\overline{A}},$ that
is, the annihilator of a right $A$-moduloid $\overline{M}$ is
contained in the annihilator of a right $\overline{A}$-module
$\overline{M}.$ Since the classical Brown--McCoy radical of a ring
equals the intersection of annihilators of all simple right
modules over that ring \cite{ar}, we have $G_l(A)\subseteq
G(\overline{A}).$ Therefore
$G_l(A)\subseteq G(\overline{A})\cap A.$\\
Conversely, let $a\in G(\overline{A})\cap A$ and let $M$ be an
arbitrary simple right $A$-moduloid. We claim that $a\in(0:M).$
Suppose $a\notin(0:M).$ Then there exists $x\in M$ such that
$xa\neq0.$ This implies $M\langle a\rangle\neq0.$ Since $M$ is
simple, there exists $b\in\overline{\langle a\rangle}$ such that
$mb=m$ for all $m\in M.$ Particularly, $xb=x.$ For all
$\bar{a}\in\overline{A},$ we have $x\bar{a}=xb\bar{a},$ which
implies $\bar{a}-b\bar{a}\in(0:x).$ Since $b\in\overline{\langle
a\rangle},$ and $a\in G(\overline{A}),$ we have $b\in
G(\overline{A}).$ Therefore there exist $w,$ $y_i$
$z_i\in\overline{A}$ and a natural number $n$ such that
$b=bw-w+\sum_{i=1}^ny_i(bz_i-z_i)$ (see, for instance,
\cite[Theorem 4.8.2]{gw}). However, as we have seen, for all
$\bar{a}\in\overline{A},$ we have $\bar{a}-b\bar{a}\in(0:x),$ that
is, $b\bar{a}-\bar{a}\in(0:x).$ Particularly, $bw-w,$
$bz_i-z_i\in(0:x),$ and therefore $b\in(0:x),$ which implies
$0=xb=x,$ a contradiction. Hence $a\in(0:M),$ and so
$G(\overline{A})\cap A\subseteq G_l(A).$ Therefore
$G_l(A)=G(\overline{A})\cap A.$
\end{proof}
\begin{remark}
Let $A$ be a regular anneid. Since $G_l(A)\subseteq G(A),$ the
previous theorem implies that $\bigoplus_{\delta\in
\Delta^*}G(\overline{A})\cap A(\delta)\subseteq\overline{G(A)},$
that is, the largest homogeneous ideal contained in the classical
Brown--McCoy radical of the ring $\overline{A}$ is contained in
the graded Brown--McCoy radical of $\overline{A}$ (contrast with
\cite[Proposition 3.5]{bs}).
\end{remark}
\noindent If $R=\bigoplus_{s\in S}R_s$ is a group-graded ring, and
if $e$ is the neutral element of $S,$ then we know from \cite{grz}
that the Brown--McCoy radical of $R_e$ is contained in the
Brown--McCoy radical of $R.$ Here we have the following result.
\begin{theorem}
Let $A$ be a regular anneid, $\overline{A}$ its linearization, and
$\epsilon$ an idempotent element of $\Delta^*.$ If $G(A)=G_l(A),$
then the classical Brown--McCoy radical $G(A(\epsilon))$ of
$A(\epsilon)$ is contained in the classical Brown--McCoy radical
$G(\overline{A})$ of $\overline{A}.$
\end{theorem}
\begin{proof}
Since by assumption, $G(A)=G_l(A),$ and since by Theorem
\ref{theoreminc}, $G(A(\epsilon))\subseteq G(A)\cap A(\epsilon),$ we have
$G(A(\epsilon))\subseteq G_l(A).$ On the other hand, Theorem
\ref{theorem} tells us that $G_l(A)=G(\overline{A})\cap A,$ and so
$G(A(\epsilon))\subseteq G(\overline{A}).$
\end{proof}

\noindent Emil Ili\'{c}-Georgijevi\'{c}\\
University of Sarajevo\\
Faculty of Civil Engineering\\
Patriotske lige 30, 71000 Sarajevo\\
Bosnia and Herzegovina\\
e-mail: emil.ilic.georgijevic@gmail.com
\end{document}